\documentclass[11pt,a4paper]{amsart}
\usepackage[margin=2.5cm]{geometry}

\usepackage[utf8]{inputenc}
\usepackage[T1]{fontenc}
\usepackage[foot]{amsaddr}
\usepackage{amsmath, amssymb}
\usepackage{amsthm}
\usepackage{mathtools}
\usepackage{algorithm}
\usepackage{algpseudocode}
\usepackage{bbm}
\usepackage{svg}
\usepackage{thm-restate}
\usepackage{biblatex}
\usepackage{url}

\usepackage{pgfplots}
\usepackage{amsmath}
\pgfplotsset{compat=1.18}

\usepackage{orcidlink}

\newcounter{common}

\newtheorem{thm}[common]{Theorem}

\newtheorem{lem}[common]{Lemma}
\newtheorem{prop}[common]{Proposition}

\newtheorem{defi}[common]{Definition}
\newtheorem{assumption}[common]{Assumption}
\newenvironment{definition}
{%
	\pushQED{\qed}\begin{defi}}
	{\popQED\end{defi}}
\theoremstyle{remark}

\newtheorem{rem}[common]{Remark}

\def\eps{{\epsilon}}

\newcommand{\kl}{\left(}
\newcommand{\kr}{\right)}
\newcommand{\el}{\left[}
\newcommand{\er}{\right]}

\newcommand{\N}{\mathbb{N}}

\newcommand{\intd}[1]{\, \mathrm{d} #1}

\newcommand{\norm}[1]{\left\lVert#1\right\rVert}

\newcommand{\E}{\mathbb{E}}
\newcommand{\R}{\mathbb{R}}

\newcommand{\abs}[1]{\left|#1\right|}
\newcommand{\betrag}[1]{\left|#1\right|}

\DeclareMathOperator{\rademp}{\widehat{\textup{Rad}}}

\allowdisplaybreaks

\DeclareMathOperator*{\argmin}{arg\,min}

\usepackage{xcolor}

\definecolor{CarmineRed}{rgb}{0.585938, 0, 0.09375}
\definecolor{UltraMarine}{rgb}{0.0703125, 0.0390625, 0.558594}

\usepackage{hyperref}
\hypersetup{colorlinks,linkcolor={CarmineRed},citecolor={UltraMarine},urlcolor={UltraMarine} }

\usepackage[capitalize]{cleveref}

\crefname{rem}{Remark}{Remarks}
\crefname{lem}{Lemma}{Lemmas}
\crefname{thm}{Theorem}{Theorems}
\crefname{prop}{Proposition}{Propositions}
\crefname{assumption}{Assumption}{Assumptions}
\crefname{defi}{Definition}{Definitions}

\title[Non-Asymptotic Analysis of Projected Gradient Descent for PINNs]{Non-Asymptotic Analysis of Projected Gradient Descent for Physics-Informed Neural Networks}

\author{Jonas Nie\MakeLowercase{{\ss}}en\textsuperscript{1,\orcidlink{0009-0004-3848-9873}}}
\email{niessen@eddy.rwth-aachen.de}

\author{Johannes M\"{u}ller\textsuperscript{2,\orcidlink{0000-0001-8729-0466}}}
\email{johannes.christoph.mueller@tu-berlin.de}

\address{$^1$ Department of Mathematics, RWTH Aachen University, Aachen, 52062, Germany}
\address{$^2$ Institut f\"{u}r Mathematik, Technische Universit\"{a}t Berlin,  Berlin,  10623, Germany}

\begin{document}
\begin{abstract}~
In this work, we provide a non-asymptotic convergence analysis of projected gradient descent for physics-informed neural networks for the Poisson equation. 
Under suitable assumptions, we show that the optimization error can be bounded by $\mathcal{O}(1/\sqrt{T} + 1/\sqrt{m} + \eps_{\textup{approx}})$, where $T$ is the number of algorithm time steps, $m$ is the width of the neural network and $\eps_{\textup{approx}}$ is an approximation error. 
The proof of our optimization result relies on bounding the linearization error and using this result together with a Lyapunov drift analysis. 
Additionally, we quantify the generalization error by bounding the Rademacher complexities of the neural network and its Laplacian. Combining both the optimization and generalization results, we obtain an overall error estimate based on an existing error estimate from regularity theory.

\end{abstract}
\maketitle
\section{Introduction}

Partial differential equations (PDEs) are foundational in describing various physical phenomena, governing processes across fields such as fluid dynamics, electromagnetism, and quantum mechanics. However, solving PDEs in high-dimensional spaces is a significant computational challenge, where traditional numerical methods, including finite element and finite difference techniques, often become infeasible due to the \emph{curse of dimensionality}.

In recent years, neural network-based methods have emerged as promising alternatives for solving high-dimensional PDEs. Among these, techniques such as the deep Ritz method, deep Galerkin approaches, physics-informed neural networks (PINNs), backward stochastic differential equation BSDE-based methods, and neural operators have shown notable success, and we refer to the overview articles~\cite{weinan2021algorithms,blechschmidt2021three, gonon2024overview} for a more detailed discussion of the different approaches. PINNs, in particular, have gained popularity due to their flexibility in incorporating physical measurements and straightforward implementation~\cite{raissi2019physics}. These properties make them well-suited for various applications, from modeling complex fluid flows to addressing inverse problems in engineering and beyond~\cite{cai2021physics, karniadakis2021physics, cuomo2022scientific}. 

Despite their immense popularity, the theoretical analysis remains difficult. Although some approximation results and error analyses exist, PINN optimization has hardly been studied. However, the optimization seems to be an essential part towards an understanding of the behaviour of PINNs. Hence, we mainly focus on this part of the problem. 
\subsubsection*{Contributions}
In this work, we analyze the optimization and generalization properties of projected gradient descent for PINNs with one hidden layer. 
Our results are non-asymptotic and do not require any assumptions on the network size thus dropping the usual strong over-parameterization assumption. 
More precisely, we show the following: 
\begin{itemize}
    \item We provide a general optimization guarantee for PINNs trained via projected gradient descent for the case of the Poisson equation, see \cref{thm:optimization bound}, showing $\mathcal{O}(1/\sqrt{T} + 1/\sqrt{m})$ convergence rate up to a remainder term stemming from a function approximation quantity. 
    Here, $T$ is the number of algorithm steps, and $m$ denotes the width of the neural network. 
    In contrast to previous works, we work in the under-parameterized regime, which for PINNs -- in contrast to a supervised learning setting -- is more suitable than the over-parameterized regime.
    \item We estimate the Rademacher complexities of the neural network and its Laplacian which allows us to bound the generalization error of the best-iterate weights produced by the gradient descent algorithm, see \cref{thm:generalization bounds}. 
    \item In \cref{thm:overall-error}, we combine the optimization result with the generalization result to get an overall error bound by using a classic result from regularity theory.
\end{itemize}

\subsubsection*{Related works}
Optimization guarantees under strong over-parametrization assumptions have been provided for various training algorithms for physics-informed neural networks. 
For shallow networks of width $\Omega(n^2)$, the empirical loss for a fixed data set has been shown to converge exponentially fast under gradient flow and gradient descent~\cite{luo2020two, gao2023gradient}. 
In this setting, the smallest eigenvalue $\lambda_0$ of the neural tangent kernel (NTK) yields the exponential convergence rate $\mathcal O((1-\lambda_0\eta)^k)$ like in the case of supervised learning. 
Further, a variant of gradient descent obtained from an implicit Euler-discretization of gradient flow has been shown to converge at the same rate~\cite{xu2024convergence}. 
In contrast, preconditioning with a Gauß-Newton or Hessian matrix with damping yields an exponential convergence rate $\mathcal O((1-\eta)^k)$ independent of the spectrum of the NTK~\cite{rathore2024challenges, xu2024convergenceNG}. 
In this context, the over-parametrization assumption is formulated in terms of the number of discretization points of the integrals appearing in the loss function of the physics-informed neural network. More explicitly, in the over-parametrized regime, the discretization points are usually fixed at the beginning of training, and the width of the neural network is chosen much larger than the number of discretization points. However, in practice, it is common to generate new integration points during optimization.
For an overview of different sampling strategies, we refer to ~\cite{toscano2024pinnspikansrecentadvances}.
In particular, the number of integration points consistently grows throughout the optimization process. 
Since in practice it is common to choose a fixed neural network size, the number of integration points naturally exceeds the network width at some point of the optimization procedure. To reflect this, we do not make any assumptions on the network size and allow new samples to be drawn at every iteration.
Besides the optimization guarantees, the convergence of the infinite width limit of the neural tangent kernel for PINNs has been investigated in~\cite{bonfanti2024challenges, zeng2024featuremappingphysicsinformedneural}. 

Apart from the theoretical treatment of the optimization of PINNs, error estimates based on the regularity theory of PDEs have been provided in a variety of settings~\cite{Mueller2022, de2022error, mishra2023estimates, de2024error}, and we refer to~\cite{Zeinhofer2024} for a unified framework. 
These estimates break down the overall error into components of optimization, approximation, and generalization. The generalization error in the context of PINNs has been analyzed, where~\cite{luo2020two, xu2024convergence} used the path norm of the networks in order to control the Rademacher complexity of the network and its Laplacian. 
For us, it is natural to work with the norm used in the projection step of our algorithm rather than the path norm. 
Further, \cite{lau2024pinnacle} bound the Rademacher complexity in dependence on the smallest eigenvalue of the NTK for over-parameterized networks. 

Next to PINNs, the optimization error of the so-called deep Ritz method, which utilizes the variational formulation of PDEs, has been analyzed in~\cite{jiao2024drm} using an approach similar to ours. However, the deep Ritz loss is fundamentally different from the PINN loss as it does not involve second-order derivatives. 

In the context of supervised learning, the over-parameterization assumption that yields exponential convergence~\cite{arora2019fine,du2019gradient} was removed in~\cite{Telgarsky2020} using a Lyapunov drift analysis to obtain a convergence guarantee up to an approximation error for functions in the neural tangent kernel RKHS. 
Subsequently, this technique was used in the context of reinforcement learning~\cite{Cayci2023, alfano2024novel}, where we extend this approach to PINNs. 

\subsubsection*{Notation}
For $k\in\mathbb N$, we denote the set of all $k$ times continuously differentiable functions on an open set $\Omega\subseteq\mathbb R^d$ by $C^k(\Omega)$. 
We denote the functions in $C^k(\Omega)$ that are Hölder continuous with exponent $\beta\in[0,1)$ by $C^{k,\beta}(\Omega)$. 
A domain $\Omega \subseteq \R^d$ is called $C^{k,\beta}$ domain if its boundary is locally the graph of a $C^{k,\beta}$ function. 
We denote the Sobolev space of $k$ times weakly differentiable functions in $L^p(\Omega)$ by $W^{k,p}(\Omega)$. We denote $\sup_{x\in \Omega} |\cdot(x)|$ by $\norm{\cdot}_{\infty(\Omega)}$. By $\N_e$ we denote the even natural numbers, i.e. $\N_e \coloneqq \{n\in \N \mid n \mod 2 = 0\}$. If $f, g$ are real functions on $\Omega \subset \R^d$, we write $f\precsim g$ if and only if there exists a constant $c>0$ such that $f(x)\leq cg(x)$ for all $x\in \Omega$. For $M\in \R^{m\times d}$ we set $B_{2,\infty}(M, r) \coloneqq \{ \tilde{M}\in \R^{m\times d} \mid \tilde{M_i} \in B_2(M_i, r) \text{ for } i\in \{1,\ldots, m\}\}$, where $B_2(M_i, r) \coloneqq \{\tilde{M_i} \in \R^d \mid \lVert{\tilde{M_i} - M_i}\rVert_2 \leq r\}$.

\section{Preliminaries on Physics-Informed Neural Networks}
This section is devoted to stating the problem, introducing physics-informed neural networks (PINNs), and introducing the algorithm we analyze in this work. 

Let us consider the Poisson equation 
\begin{equation}\label{eq:poisson with zero boundary}
			\begin{aligned}
				-\Delta u&=f && \text{in }\Omega \subset \R^d \\
				u&=0 && \text{on }\partial\Omega
			\end{aligned},
\end{equation}
which is the prototypical elliptic PDE. 
It is well known that for regular domains and data, meaning $C^{2, \beta}$-domain $\Omega\in \R^d$ and  $f\in C^\beta(\Omega)$ for some $\beta>0$, there exists a unique solution $u^\ast\in C^{2,\beta}(\Omega)$ of this problem~\cite{Evans2010}, which can be characterized as 
\begin{align}\label{eq:pinn energy}
    u^* &= \argmin_{u\in C^2(\Omega)} \int_{\Omega} \kl \Delta u(x) + f(x)\kr^2 \intd{x} + \lambda \int_{\partial \Omega} u(s)^2 \intd{s}
\end{align}
for any $\lambda>0$. 
Note that $u^\ast$ is the unique element achieving a value of zero of the objective on the right-hand side of \Cref{eq:pinn energy}. 

\subsubsection*{Physics-informed neural networks (PINNs)}
To numerically solve \eqref{eq:poisson with zero boundary}, classical approximation schemes like Finite Element Methods (FEMs) work extremely well for low dimensional settings. 
However, they suffer from the curse of dimensionality, and hence these methods become computationally infeasible in high-dimensional settings or in complex geometries. 
However, one can use the formulation~\eqref{eq:pinn energy} of the Poisson equation as a least squares problem in combination with any parametric function class $\mathcal F = \{ F(\cdot;\theta) \colon\overline{\Omega} \to \mathbb R \mid \theta \in\Theta \}$ with parameter space $\Theta$, like finite element functions or neural networks~\cite{bochev1998finite, raissi2019physics,sirignano2018dgm}. 
This leads to the loss function 
\begin{equation}\label{eq:loss_function_exact}
		E(\theta) = \int\limits_\Omega \kl \Delta F(x;\theta) + f(x)\kr^2 \intd{x} + \lambda \int\limits_{\partial \Omega} F(y;\theta)^2 \intd{y},
\end{equation}
with a boundary penalty parameter $\lambda > 0$. 
Intuitively, minimization of $E$ should lead to a reasonable approximation $F(\cdot; \theta)$ of $u^*$. 
Indeed, it is a well-known fact of regularity theory that 
\begin{align}
    \lVert u^* - F(\cdot; \theta) \rVert_{H^{1/2}(\Omega)} \le c_{\textup{reg}} \cdot \sqrt{E(\theta)}, 
\end{align}
where $c_{\textup{reg}}>0$ denotes a suitable regularity constant, see \cite{Evans2010, grisvard2011elliptic, Mueller2022, Zeinhofer2024}. 

When neural networks are trained with a least-square loss of the form \eqref{eq:loss_function_exact} this is commonly referred to as \emph{Physics-Informed Neural Networks (PINNs)} or \emph{deep Galerkin method}~\cite{raissi2019physics,sirignano2018dgm} and has gained immense popularity in recent years. 
One benefit of this ansatz is that it can easily be implemented even for complex geometries and high-dimensional problems as long as one can discretize the integrals in the definition of the objective~\eqref{eq:loss_function_exact}. Since it is typically not feasible to evaluate the occurring integrals analytically, we approximate them using a quadrature rule, which can be deterministic or stochastic. 
Thus, we define the empirical loss
\begin{equation}\label{eq:opt}
  {E}_S(\theta) \coloneqq \frac{1}{b_\Omega} \sum_{j=1}^{b_\Omega} \kl \Delta F(x_j; \theta) + f(x_j)\kr^2 + \frac{\lambda}{b_{\partial\Omega}} \sum_{j=1}^{b_{\partial \Omega}}  F(y_j; \theta)^2, 
\end{equation}
where $S = ((x_j)_{j=1}^{b_\Omega},(y_j)_{j=1}^{ b_{\partial\Omega}})\subseteq \Omega\times\partial\Omega$ denotes the set of quadrature points. 

\subsubsection*{Neural network model and training algorithm}
In general, any neural network architecture  can be used for PINNs, however, we focus our analysis on fully-connected feedforward neural networks with one hidden layer, which are given by 
\begin{align}
    F(x; \theta) \coloneqq F(x;\theta,c) \coloneqq \sum_{i=1}^m c_i \sigma(\theta_i^T x),
\end{align}
where $\sigma\colon\mathbb R\to\mathbb R$ is a smooth activation function, $m\in \N_e$ is an even number, and \emph{hidden weights} $\theta_i\in\R^d$ and \emph{output weights} $c_i\in\R$ for $i = 1, \dots, m$. We fix the output weights $c_i$ and consider the projected gradient descent described in \Cref{alg:pinn optimization}.
\begin{algorithm}
		\caption{$\norm{\cdot}_{2,\infty}$-Projected (Stochastic) Gradient Descent}
		\begin{algorithmic}
				\Require projection radius $p>0$, learning rate $\eta >0$,
                and
                initialization $(\theta(0),c)$ 
				\For{$t = 0,\ldots, T-1$}
                    \State update integration points $S_t$ 
                    \Comment{might depend on $S_{t-1}$} 
					\For{$i=1,\ldots, m$}
						\State $\widetilde{\theta_i}(t+1) = \theta_i(t) - \eta \nabla_{\theta_i} E_{S_t}(t,\theta(t))$
						\State $\theta_i(t+1) = \pi_{{B_2}\kl \theta_i(0), \frac{p}{\sqrt{m}}\kr} \kl \widetilde{\theta_i}(t+1)\kr$
					\EndFor{}
				\EndFor{}
		\end{algorithmic}
		\label{alg:pinn optimization}
\end{algorithm}
\begin{rem}
    The projection onto a closed convex subset of $\R^d$ is one of the factors that makes the proof of our optimization result \cref{thm:optimization bound} work. 
    However, our optimization result can still be used for non-projected gradient descent by making the following observation: Let $\tilde{\theta}$ be the random variable that is defined as the maximum overall weights from $0$ to $T-1$ with respect to the $\norm{\cdot}_{2,\infty}$ norm when running gradient descent without the projection step under random initialization. Then we choose $p$ large enough to ensure that $\lVert{\tilde{\theta}}\rVert_{2,\infty} < p$ with high probability. Nevertheless, choosing a suitable $p$ is a nontrivial task, which we will not focus on in this work.
\end{rem}
Here, $\pi_A(x)$ denotes the orthogonal projection of $x\in \R^d$ onto the subset $A\subset \R^d$ with respect to the Euclidean inner product. We initialize the neural network at time $t=0$ by setting
\begin{equation}\label{eq: ntk initialization}
    \begin{aligned}
        \theta_i(0) &= \theta_{i+\frac{m}{2}}(0) &&\overset{\mathclap{\text{iid}}}{\sim} \mathcal{N}_a(0,I_d)\\
        c_i &= - c_{i+\frac{m}{2}} &&\overset{\mathclap{\text{iid}}}{\sim} \frac{1}{\sqrt{m}}\text{Rademacher}
    \end{aligned}
\end{equation}
for $i=1, \ldots, \frac m2$, where $\mathcal{N}_a(0, I_d)$ denotes the $d$-dimensional symmetrically truncated normal distribution with support in $[-a,a]^d$, see \cref{definition: truncated normal distribution}.
In principle, our results work with any distribution that has compact support. In this work, we fix the distribution $\mathcal{N}_a(0,I_d)$ because it works well in practice.

We consider the case where the output weights $c_i$ are kept as initialized, and only the hidden weights $\theta_i$ are trained, which is common in the neural tangent kernel literature, see \cite{Chizat2018,Du2019}. 
Thus, we usually write $F(x; \theta)$ instead of $F(x; \theta,c)$.

The specific initialization described above as well as the projection after each step within the algorithm, may seem artificial. Both restrictions are made to simplify the analysis of the optimization procedure. It leads to an analysis within a special regime, which we further elaborate in the beginning of the next section.
An important consequence of the symmetric initialization is that then $F(x; \theta(0)) = 0$ for all $x\in \R^d$, which simplifies the convergence analysis.

Regarding the optimization result, the integration points can be chosen arbitrarily, since our optimization theorem does not depend on that choice. However, to achieve a reasonably small generalization error, it is common to use integration points that arise from a uniform distribution on $\Omega$ or $\partial \Omega$, respectively. Later on, we will analyse the generalization error as well with respect to such a uniform distribution.

\section{Optimization Guarantee for Projected GD for PINNs}
In this section, we provide a non-asymptotic optimization guarantee on the empirical loss for the projection gradient descent algorithm described in \Cref{alg:pinn optimization}. Informally, we show that 
\begin{align}
    \frac{1}{T}\sum_{t=0}^{T-1} E_{S_t}(\theta(t)) \precsim \frac{1}{\sqrt T} + \frac{1}{\sqrt m} + \epsilon_{\textup{approx}}(u^*, p), 
\end{align}
with high probability over the initialization, where $p>0$ is the projection radius and $\epsilon_{\textup{approx}}(u^*, p)$ the approximation error of $u^*$ with networks of arbitrary size that are contained in the ball with radius $p$ around the initialization. 

In the context of PINNs, the integration points in the numerical discretization of the integrals of the continuous loss function~\eqref{eq:loss_function_exact} play the role of data points. 
Hence, data points are cheap to generate, and it is common to continuously sample new points~\cite{raissi2019physics, cuomo2022scientific}. 
Thus, PINNs naturally operate in an underparametrized regime. 
Where existing results guarantee exponential convergence in the quadratically over-parametrized regime, our result is located in the non-over-parametrized setting. We pay for this with a rate of $\frac{1}{\sqrt{T}}$ instead of the usual exponential convergence rate~\cite{Du2019,gao2023gradient} one obtains under quadratic over-parametrization with a fixed data set. 

The main steps to obtain our optimization result consist in bounding a linearization error -- which is possible because of the projection in \cref{alg:pinn optimization} -- using the linearized version to make statements about the convergence properties, and carrying on the function approximation error throughout the proof.

\subsection{Neural Tangent Kernels and Reproducing Kernel Hilbert Spaces} 
In the proof of our main result, we compare the training dynamics of the network to the training dynamics of a linearized model. Therefore, we consider the \emph{linearized neural network} around the initialization and denote it by 
\begin{align}\label{eq:linearization}
        F^{\text{lin}}(x; \theta) &\coloneqq \underbrace{F(x; \theta(0))}_{=0} + \sum_{i=1}^m \nabla_{\theta_i}^T F(x; \theta(0)) \el \theta_i - \theta_i(0)\er.
\end{align}
Note that $F^{\textup{lin}}$ is a linear function of $\theta$ and thus can be written as a linear model with features $\Phi_i(x) = \nabla_{\theta_i} F(x; \theta(0))$. 
These features depend on the random initialization and thus $F^{\textup{lin}}$ is a \emph{random feature model} introduced by \cite{rahimi2007random} with features 
\begin{align}\label{eq:NTK-features}
    \Phi_{\textup{NTK}}(x;\theta) \coloneqq \nabla_\theta\sigma(\theta^T x) = x \sigma'(\theta^T x). 
\end{align}
Here, NTK stands for \emph{neural tangent kernel} as the features come from a linearization of the neural network model, which spans the tangent space of the nonlinear model in function space, see~\cite{jacot2018neural}. 

Associated with a random feature model with feature map $\Phi$ is the Hilbert space 
\begin{align}
    \mathcal H_{\textup{RF}} \coloneqq \left\{ u\colon\R^d\to\R \mid u(x) = 
    \mathbb E_{\theta} \el v(\theta)^T \Phi(x;\theta) \er \textup{ with } \E_{\theta} \el\lVert v(\theta) \rVert_2^2\er < \infty \right\}, 
\end{align}
where we call $v$ the \emph{transport mapping} associated to $u$, see \cite{rahimi2008uniform, ji2019neural}. 
The inner product in $\mathcal{H}_{\textup{RF}}$ is given by 
\begin{align}
    \langle u_1, u_2\rangle_{\mathcal H_{\textup{RF}}} \coloneqq \E_{\theta} \el v_1(\theta)^T v_2(\theta) \er,
\end{align}
where $v_i$ denotes the transport mapping of $u_i$. 
This forms a reproducing kernel Hilbert space (RKHS) with the kernel 
\begin{align}
    K_{\textup{RF}}(x,x') \coloneqq \E_{\theta} \el\Phi(x;\theta)^T \Phi(x';\theta)\er.
\end{align}
For the random features defined in~\eqref{eq:NTK-features} the corresponding kernel is known as the neural tangent kernel (NTK) and given by 
\begin{align}
     K_{\text{NTK}}(x, x') = x^T x' \cdot \E_{\theta} \el \sigma'(\theta^T x) \sigma'(\theta^T x') \er
\end{align}
and was first introduced in \cite{jacot2018neural} to describe the learning dynamics of neural network training and subsequently used to establish exponential convergence under massive over-parametrization~\cite{du2019gradient, arora2019fine}. 
We refer the interested reader to the excellent overview and introduction to NTK theory provided in~\cite{bowman2023spectral}. 
We refer to the corresponding reproducing kernel Hilbert space as \emph{neural tangent RKHS}, which is given by 
\begin{align}
    \mathcal H_{\textup{NTK}} = \left\{ u\colon \mathbb R^d\to\mathbb R \mid u(x) = \mathbb E_{\theta}[\Phi_{\text{NTK}}(x;\theta)^T v(\theta)], \text{ where } \mathbb E_{\theta}[\lVert v(\theta) \rVert_2^2]<+\infty\right\},
\end{align}
see also \cite{rahimi2008uniform} for a general discussion of random feature RKHS.

To see the connection of the linearization~\eqref{eq:linearization} to the expression of functions in the neural tangent RKHS via a transport function, we note that 
\begin{align}
    F^{\textup{lin}}(x;\theta) = \frac{1}{m}\sum_{i=1}^m \Phi_{\textup{NTK}}(x;\theta_i(0))^T v_i,
\end{align}
where $v_i =c_i \sqrt{m} \cdot(\theta_i - \theta_i(0))$. 
In the projected gradient descent described in \Cref{alg:pinn optimization} we enforce $\lVert\theta_i-\theta_i(0)\rVert_2\le \frac{p}{\sqrt{m}}$ and thus we have $\lVert v_i \rVert_2 \le p$. 
Hence, the functions that can be expressed by the linearized neural network with the parameters satisfying $\lVert \theta_i - \theta_i(0)\rVert_2\le \frac{p}{\sqrt{m}}$ can be interpreted as discretizations of functions contained in the class
\begin{align}
    \mathcal{F}^p \coloneqq \left\{
        u\in\mathcal H_{\textup{NTK}} \mid  u(x) = \mathbb E_{\theta}[\Phi_{\text{NTK}}(x;\theta)^T v(\theta)], \text{ where } \sup_{\theta\in \R^d} \norm{v(\theta)}_2 \leq p 
    \right\} 
\end{align}
of functions with bounded transport mapping. 
Conversely, for a bounded transport function with $\norm{v(\theta)}_2\le p$, we can construct parameters 
\begin{align}\label{eq:optimal weights}
    \theta_i \coloneqq \theta_i(0) + \frac{c_i}{\sqrt{m}} \cdot v(\theta_i(0))
\end{align}
such that $\lVert \theta_i - \theta_i(0) \rVert_2\le \frac{p}{\sqrt{m}}$.

\subsection{Linearization and Approximation Errors}
Here, we bound the linearization error and the approximation error of the linearized model. We defer all proofs to the appendix.

We make the following assumption on the regularity of the activation. 

\begin{assumption}\label{ass:smoothness}
    The activation function satisfies $\lVert{\sigma^{(k)}}\rVert_\infty \leq \sigma_k$ for $k=1, \dots, 4$. 
\end{assumption}
\begin{rem}
    \Cref{ass:smoothness} allows for unbounded activations, as long as the first four derivatives are globally bounded. 
    We require this in our theoretical analysis for various estimates in the linearization error and in the proof of our optimization bound. 
    In particular, the hyperbolic tangent, the sigmoid function, sine and cosine as well as smoothed ReLUs (e.g. Swish functions) satisfy \Cref{ass:smoothness}.  
    Note that powers of the ReLU activation do not satisfy \Cref{ass:smoothness} as they have unbounded derivatives. 
    Whereas those activations arise in theoretical works~\cite{gao2023gradient,luo2020two, xu2024convergenceNG} they are not commonly used in practice~\cite{wang2023expert, ABBASI2024128352, CiCP-34-4, Maczuga23}. 
\end{rem}
\begin{lem}[Linearization error]\label{lemma:linearization error}
Let $m\in \N_e$, fix $\theta(0)$ and let $\theta_i\in {B_2}(\theta_i(0), \frac{p}{\sqrt{m}})$ for all $i\in\{1, \ldots, m\}$, then there exists a constant $C = C(\sigma_2, \sigma_3, \sigma_4)>0$ such that for all $x\in \R^d$ we have
        \begin{align}
            \betrag{\Delta F\kl x; \theta\kr- \sum_{i=1}^m \nabla_{\theta_i}^T \Delta F\kl x; \theta(0) \kr \el \theta_i - \theta_i(0)\er} &\leq C\frac{\norm{x}_2^2 p^2}{\sqrt{m}}  \quad \text{and} \label{eq:linearization error laplace}\\
            \betrag{F\kl x; \theta\kr- \sum_{i=1}^m \nabla_{\theta_i}^T F\kl x; \theta(0) \kr \el \theta_i - \theta_i(0)\er} & \le
            \frac{\sigma_2\norm{x}^2_2 p^2}{\sqrt{m}}. 
            \label{eq:linearization error}
        \end{align}
\end{lem}
Next, we bound the approximation error of a finite-width neural network and its Laplacian under random initialization. 

\begin{lem}[Approximation error]\label{lemma:approximation error}
Let $m\in \N_e$, fix $\theta(0)$ initialized according to \eqref{eq: ntk initialization}, fix a projection radius $p>0$, let $u\in \mathcal{F}^p$ be fixed with corresponding transportation mapping $v$ and define $\theta_i$ accordingly as in \eqref{eq:optimal weights}. Let $\delta\in(0,1)$ and set $\eps \coloneqq u^* - u$. 
Then the following statements hold: 
\begin{enumerate}
    \item There exists a constant $C = C(a, \sigma_2, \sigma_3)>0$ such that for all $x\in \Omega$ with probability at least $1-\delta$ we have
    \begin{align}
        \betrag{\Delta u^*(x)- \sum_{i=1}^m \nabla_{\theta_i}^T \Delta F\kl x; \theta(0) \kr \el \theta_i - \theta_i(0)\er} \leq C \frac{\sqrt{\log \kl \frac{1}{\delta}\kr }}{\sqrt{m}} + \abs{\Delta\eps(x)}. 
    \end{align}
    \item There exists a constant $C = C(a, \sigma_1)>0$ such that for all $x\in \partial\Omega$ with probability at least $1-\delta$ we have 
     \begin{align}
         \betrag{u^*(x)- \sum_{i=1}^m \nabla_{\theta_i}^T F\kl x; \theta(0) \kr \el \theta_i - \theta_i(0)\er} \leq C \frac{\sqrt{\log \kl \frac{1}{\delta}\kr }}{\sqrt{m}} + \abs{\eps(x)}. 
     \end{align}
\end{enumerate}
\end{lem}

Now we have collected the auxiliary results for the proof of the main theorem. 

\subsection{Finite-Time Bound for Projected Gradient Descent} 
With that, we can formulate our main optimization guarantee.

\begin{thm}[Performance of Projected Gradient Descent]\label{thm:optimization bound}
Consider a bounded $C^{2,\beta}$ 
domain $\Omega\subseteq\mathbb R^d$ and set $c_\Omega\coloneqq\sup\{\lVert x \rVert \mid x\in\Omega\}$ and let $u^\ast\in C^{2,\beta}(\Omega)\cap C(\overline{\Omega})$ be the solution of the Poisson equation~\eqref{eq:poisson with zero boundary} for some $f\in C^{\beta}(\Omega)$ and $\beta>0$. Further, let $\theta(0)$ be a fixed sample according to the symmetric initialization \eqref{eq: ntk initialization} where the inner weights are sampled from a truncated Gaussian
with truncation parameter $a>0$, see~\Cref{definition: truncated normal distribution}. We run \cref{alg:pinn optimization} for $T$ time steps with batch sizes $b_\Omega, b_{\partial \Omega}$, arbitrary integration sets $(S_t)_{t=0}^{T-1}$, step size $\eta = \frac{1}{\sqrt{T}}$, some $p>0$, and penalization strength $\lambda>0$. 
Let \Cref{ass:smoothness} hold and set 
\begin{align}
    \eps_p \coloneqq \inf_{u\in \mathcal{F}^p} \left( 
    \norm{u-u^\ast}_{\infty(\partial \Omega)} + \norm{\Delta u-\Delta u^\ast}_{\infty(\Omega)}
    \right). 
\end{align}
Then, there exist constants $C_1, C_2>0$ and $C_3>0$ depending polynomially on the quantities $a, d, \lambda, c_{\Omega}, p, \sigma_1, \sigma_2, \sigma_3$ and $\sigma_4$ such that for any $\delta \in (0,1)$ it holds that
\begin{align}
    \frac{1}{T} \sum_{t=0}^{T-1} E_{S_t}(\theta(t)) \leq \frac{C_1}{\sqrt{T}} + \frac{C_2 \sqrt{\log \kl \frac{\max \{b_\Omega, b_{\partial \Omega}\}}{\delta}\kr }}{\sqrt{m}}+ C_3 \eps_{p}
\end{align}
with probability at least $1-2\delta$ over the random initialization.
\end{thm}
\begin{rem}[Realizability]
    If the solution of \eqref{eq:poisson with zero boundary} is realizable, meaning that there exists an $\alpha > 0$ such that $u^* \in \mathcal{F}^\alpha$, then the approximation error $\eps_{\text{approx}}(u^*, p)$ is $0$ for any $p\geq \alpha$. Since $C_1$ and $C_2$ depend polynomially on $p$, there is a tradeoff between the approximation error and the linearization error. With prior knowledge on $\alpha$ it is possible to choose $p$ as small as possible to avoid unnecessarily large $m$ and $T$ to achieve a prescribed accuracy.
\end{rem}
\cref{thm:optimization bound} shows that the optimization procedure converges with order $1/\sqrt{T} + 1/\sqrt{m}$ up to an approximation error with high probability. Since we have not yet looked at the generalization properties of $E_{S_t}$ to $E$, this optimization bound is independent of the chosen data points. 
In contrast to previous results in the NTK regime, we do not need a massive over-parametrization. This comes at the cost of a slower rate in $T$ compared to previous works.
\begin{proof}[Proof of \cref{thm:optimization bound}]
For a function $u\in C^{2}(\Omega)$, we define 
\begin{align}
        \norm{u}_{W^{0,2;\infty}} \coloneqq \norm{u}_{\infty(\partial \Omega)} + \norm{\Delta u}_{\infty(\Omega)} 
\end{align}
and for $\xi>0$ we fix a function
\begin{align*} 
    u_p^\xi \in \left\{ w\in \mathcal{F}^p \,\middle|\,
		\norm{w-u^*}_{W^{0,2;\infty}}
		< \eps_p + \xi \right\}
\end{align*} 
that achieves the optimal approximation error $\eps_p$ up to $\xi$. 
We denote the transportation mapping corresponding to $u^\xi_p$ by $v_{p,\xi}$
and set
\begin{align*}
    \eps_{p}^\xi(x)&\coloneqq u^*(x)-u_p^\xi(x). 
\end{align*}
Set $\overline{\theta}$ to be the weights according to $v_{p,\xi}$ as in \eqref{eq:optimal weights}.
Now let
    \[
        \mathcal{L}(W) = \norm{\theta-\overline{\theta}}_F^2 = \sum_{i=1}^m \norm{\theta_i - \overline{\theta_i}}_2^2
    \]
		be the potential function, where $\norm{M}_F$ denotes the Frobenius norm of the matrix $M$. We recall that $\theta_i$ denotes the weight vector of the $i$-th vector. By the nonexpansivity of the projection operator onto $B_i \coloneqq B_2 \kl \theta_i(0), \frac{p}{\sqrt{m}}\kr$ there holds 
		\begin{align*}
			\mathcal{L}(\theta(t+1)) &= \sum_{i=1}^m \norm{\theta_i(t+1) - \overline{\theta_i}}_2^2 \\
			&= \sum_{i=1}^m \norm{\pi_{B_i}\kl \widetilde{\theta_i}(t+1)\kr - \pi_{B_i}\kl \overline{\theta_i}\kr}_2^2 \\
			&\leq \sum_{i=1}^m \norm{\widetilde{\theta_i}(t+1) - \overline{\theta_i}}_2^2.
		\end{align*}
		Thus we compute
		\begin{align}\label{eq:lyapunov inequality}
			\begin{split}
			    &\quad\ \mathcal{L}(\theta(t+1)) \\
			&\leq \sum_{i=1}^m \norm{\widetilde{\theta_i}(t+1) - \overline{\theta_i}}_2^2 \\
			&= \sum_{i=1}^m \norm{\theta_i(t) - \overline{\theta_i} - \eta \nabla_{\theta_i} E_{S_t}(\theta(t))}_2^2 \\
			&= \sum_{i=1}^m \norm{\theta_i(t) - \overline{\theta_i}}_2^2 - 2\eta \underbrace{\sum_{i=1}^m \nabla_{\theta_i}^T E_{S_t}(\theta(t))\el \theta_i(t) - \overline{\theta_i}\er}_{\eqqcolon (**)} + \eta^2 \underbrace{\sum_{i=1}^m \norm{\nabla_{\theta_i} E_{S_t}(\theta(t))}_2^2}_{\eqqcolon (*)}.
			\end{split}
    \end{align}
    The goal in our proof consists of upper bounding $(*)$ and lower bounding $(**)$. To get an estimate of the average iterate $\frac{1}{T} \sum_{t=0}^{T-1} E_{S_t}(\theta(t))$, we will ensure to lower bound $(**)$ in terms of $E_{S_t}$. Before bounding $(*)$ and $(**)$ separately, we observe that
		\begin{align*}
			\nabla_{\theta_i} E_{S_t}(\theta(t)) &= \frac{2}{b_\Omega} \sum_{j=1}^{b_\Omega} \kl \Delta F(x_j^t;\theta(t)) + f(x_j^t)\kr \Delta \kl \nabla_{\theta_i} F(x_j^t;\theta(t))\kr\notag\\ 
			&\quad\ + \frac{2\lambda}{b_{\partial \Omega}} \sum_{j=1}^{b_{\partial \Omega}} F(y_j^t; \theta(t)) \nabla_{\theta_i} F(y_j^t; \theta(t)),
		\end{align*}
		where the laplace operator $\Delta$ is applied entry-wise on the vector $\nabla_{\theta_i} F(x_j^t;\theta(t))$. 
  
        \noindent
        \textbf{Upper Bounding $(*)$:} We start with an upper bound for $(*)$. We can write
		\begin{equation*}
			\Delta \nabla_{\theta_i} F(x;\theta(t)) = \frac{c_i}{\sqrt{m}} \begin{bmatrix}
				\el \sum_{l=1}^d x_1 \theta_{il}^2(t) \sigma'''\kl \theta_i(t)^T x\kr \er + 2\theta_{i1}(t) \sigma''\kl \theta_i(t)^T x\kr \\
				\vdots \\
				\el \sum_{l=1}^d x_d \theta_{il}^2(t) \sigma'''\kl \theta_i(t)^T x\kr \er + 2\theta_{id}(t) \sigma''\kl \theta_i(t)^T x\kr
			\end{bmatrix}.
		\end{equation*}
		Since
		\[
			\betrag{2 \theta_{i1}(t) \sigma''\kl \theta_i(t)^T x\kr + \sum_{l=1}^d x_1 \theta_{il}^2(t) \sigma'''\kl \theta_i(t)^T x \kr} \leq 2 \norm{\theta_i(t)}_2 \sigma_2 + c_\Omega \sigma_3 \norm{\theta_i(t)}_2^2,
		\]
		we conclude
		\[
			\norm{\Delta \nabla_{\theta_i} F(x_j^t;\theta(t))}_2 \leq \frac{d}{\sqrt{m}} \kl 2\norm{\theta_i(t)}_2 \sigma_2 + c_\Omega \sigma_3 \norm{\theta_i(t)}_2^2 \kr
		\]
		and with 
		\[
			\norm{\theta_i(t)}_2 = \norm{\theta_i(t) - \theta_i(0) + \theta_i(0)}_2 \leq \frac{p}{\sqrt{m}} + \norm{\theta_i(0)}_2
		\]
		we obtain
		\begin{align*}
			\norm{\Delta \nabla_{\theta_i} F(x_j^t;\theta(t))}_2 &\leq \frac{d}{\sqrt{m}} \el 2\kl \frac{p}{\sqrt{m}} + \norm{\theta_i(0)}_2\kr \sigma_2 + c_\Omega \sigma_3 \kl \frac{p}{\sqrt{m}} + \norm{\theta_i(0)}_2\kr^2 \er \notag\\
			&\leq \frac{d}{\sqrt{m}} \el \frac{2p\sigma_2}{\sqrt{m}} + 2\norm{\theta_i(0)}_2 \sigma_2 + \frac{2 c_\Omega \sigma_3 p^2}{m} + {2 c_\Omega \sigma_3 \norm{\theta_i(0)}_2^2} \er \notag\\
			&\leq \frac{2dp\sigma_2 + 2dc_\Omega \sigma_3 p^2}{m} + \frac{2d\norm{\theta_i(0)}_2 \sigma_2 + 2dc_\Omega \sigma_3 \norm{\theta_i(0)}_2^2}{\sqrt{m}}\notag \\
			&\leq \frac{2dp\sigma_2 + 2dc_\Omega \sigma_3 p^2}{m} + \frac{2d \theta_a \sigma_2 + 2dc_\Omega \sigma_3 \theta_a^2}{\sqrt{m}} \notag \\
			&= \frac{c_1(d,c_\Omega,p, \sigma_2, \sigma_3)}{m} + \frac{c_2(a, d,c_\Omega,p, \sigma_2, \sigma_3)}{\sqrt{m}},
		\end{align*}
		where $c_1$ and $c_2$ are constants depending on the mentioned quantities. We compute
		\begin{align*}
			\betrag{F(x;\theta(t))} &= \betrag{F(x;\theta(t)) - F(x;\theta(0))} \notag \\
			&= \betrag{\frac{1}{\sqrt{m}} \sum_{i=1}^m c_i \sigma\kl \theta_i(t)^T x\kr - \frac{1}{\sqrt{m}} \sum_{i=1}^m c_i \sigma\kl \theta_i(0)^T x\kr} \notag \\
			&= \betrag{\frac{1}{\sqrt{m}} \sum_{i=1}^m c_i \el \sigma\kl \theta_i(t)^T x\kr - \sigma \kl \theta_i(0)^T x\kr \er}\notag \\
			&\leq \frac{1}{\sqrt{m}} \sum_{i=1}^m \underbrace{\betrag{\sigma \kl \theta_i(t)^T x\kr - \sigma \kl \theta_i(0)^T x\kr}}_{\begin{aligned}
					&\leq \sigma_1 \betrag{\theta_i(t)^T x - \theta_i(0)^T x} \\
					&= \sigma_1 \betrag{x^T \kl \theta_i(t) - \theta_i(0)\kr} \\
					&\leq \sigma_1 c_\Omega \norm{\theta_i(t) - \theta_i(0)}_2\\
					&\leq \frac{\sigma_1 c_\Omega p}{\sqrt{m}}
			\end{aligned}}\notag \\
			&\leq \sigma_1 c_\Omega p.
		\end{align*}
		Further, we observe
		\begin{equation*}
			\norm{\nabla_{\theta_i} F\kl x; \theta(t)\kr} = \norm{\frac{1}{\sqrt{m}}c_i x \sigma'\kl \theta_i(t)^T x\kr}_2 \leq \frac{c_\Omega \sigma_1}{\sqrt{m}}
		\end{equation*}
		as well as
		\begin{align*}
			&\quad\, \betrag{\Delta F\kl x; \theta(t)\kr} \notag \\
			&= \betrag{\Delta F\kl x; \theta(t)\kr - \Delta F\kl x;\theta(0)\kr}\notag \\
			&= \betrag{\sum_{j=1}^d \frac{1}{\sqrt{m}} \sum_{i=1}^m c_i \el \theta_{ij}^2(t) \sigma''\kl \theta_i(t)^T x\kr - \theta_{ij}^2(0) \sigma''\kl \theta_i(0)^T x \kr \er} \notag \\
			&\leq \sum_{j=1}^d \frac{1}{\sqrt{m}} \sum_{i=1}^m \frac{1}{\sqrt{m}} \betrag{\theta_{ij}^2(t)} \betrag{\sigma''\kl \theta_i(t)^T x\kr - \sigma''\kl \theta_i(0)^T x \kr} + \sigma_2 \betrag{\theta_{ij}^2(t) - \theta_{ij}^2(0)} \notag\\ 
			&\leq c_3\kl a, d, p, c_\Omega, \sigma_2\kr,
		\end{align*}
		where we used that $\Delta F\kl x; \theta(0)\kr = 0$ for all $x\in \overline{\Omega}$. 
        From now on, we will only mention the constants and not their dependence anymore. 
        Using the previous results, we can show
		\begin{align}\label{eq:bound norm of loss function}
			\begin{split}
       \norm{\nabla_{\theta_i} E_{S_t}(\theta(t))}_2
       &
   \leq \norm{\frac{2}{b_\Omega} \sum_{j=1}^{b_\Omega} \kl \Delta F(x_j^t;\theta(t)) + f(x_j^t)\kr \Delta \kl \nabla_{\theta_i} F(x_j^t;\theta(t))\kr}_2 \\ 
			&\quad\ + \norm{\frac{2\lambda}{b_{\partial \Omega}} \sum_{j=1}^{b_{\partial \Omega}} F(y_j^t; \theta(t)) \nabla_{\theta_i} F(y_j^t; \theta(t))}_2  \\
			&\leq \frac{2}{b_\Omega} \sum_{j=1}^{b_\Omega} \kl \betrag{\Delta F(x_j^t;\theta(t))} + \betrag{f(x_j^t)}\kr \norm{\Delta \kl \nabla_{\theta_i} F(x_j^t;\theta(t))\kr}_2 \\ 
			&\quad\ + \frac{2\lambda}{b_{\partial \Omega}} \sum_{j=1}^{b_{\partial \Omega}} \betrag{F(y_j^t; \theta(t))} \norm{\nabla_{\theta_i} F(y_j^t; \theta(t))}_2  \\
			&\leq \frac{c_4 + c_5}{\sqrt{m}}. 
			\end{split}
		\end{align}
		Thus \eqref{eq:bound norm of loss function} implies
		\begin{equation}\label{eq:bound on *}
			(*) = \sum_{i=1}^m \norm{\nabla_{\theta_i} E_{S_t}(\theta(t))}_2^2 \leq (c_4 + c_5)^2.
		\end{equation} 

\noindent\textbf{Lower Bounding $(**)$:} To bound $(**)$ from below, we observe
		\begin{align*}
			&\quad\ \sum_{i=1}^m \nabla_{\theta_i} E_{S_t}(\theta(t))^T \el \theta(t)-\overline{\theta}\er \notag \\
			&= \frac{2}{b_\Omega}\sum_{j=1}^{b_\Omega} \underbrace{\kl \Delta F(x_j^t;\theta(t)) + f(x_j^t)\kr \sum_{i=1}^m \kl \Delta \nabla_{\theta_i} F(x_j^t;\theta(t))\kr^T \el \theta_i(t) - \overline{\theta_i}\er}_{\eqqcolon (\text{II})} \notag \\
			&\quad\ + \frac{2\lambda}{b_{\partial \Omega}} \sum_{j=1}^{b_{\partial \Omega}} \underbrace{F(y_j^t;\theta(t)) \sum_{i=1}^m \nabla_{\theta_i}^T F(y_j^t;\theta(t))\el \theta_i(t) - \overline{\theta_i}\er}_{\eqqcolon (\text{I})}.
		\end{align*}
		We emphasize that it is necessary to obtain a lower bound in which the empirical loss function appears again. Starting with (I), there holds
		\begin{align*}\label{eq:decomposition of I}
			&\quad\, \sum_{i=1}^m \nabla_{\theta_i}^T F(y_j^t;\theta(t))\el \theta_i(t) - \overline{\theta_i}\er\notag \\
			&= \sum_{i=1}^m \nabla_{\theta_i}^T F\kl y_j^t;\theta(t)\kr \el \theta_i(t) - \theta_i(0)\er - \sum_{i=1}^m \nabla_{\theta_i}^T F\kl y_j^t; \theta(t)\kr \el \overline{\theta_i} - \theta_i(0)\er \notag \\
			&= \sum_{i=1}^m \nabla_{\theta_i}^T F\kl y_j^t; \theta(0)\kr \el \theta_i(t) - \theta_i(0)\er - \sum_{i=1}^m \nabla_{\theta_i}^T F\kl y_j^t; \theta(0)\kr \el \overline{\theta_i} - \theta_i(0)\er \notag\\
			&\quad\, + \sum_{i=1}^m \nabla_{\theta_i}^T F\kl y_j^t; \theta(0)\kr \el \overline{\theta_i} - \theta_i(0)\er - \sum_{i=1}^m \nabla_{\theta_i}^T F\kl y_j^t; \theta(t)\kr \el \overline{\theta_i} - \theta_i(0)\er \notag\\
			&\quad\, + \sum_{i=1}^m \nabla_{\theta_i}^T F\kl y_j^t;\theta(t)\kr \el \theta_i(t) - \theta_i(0)\er - \sum_{i=1}^m \nabla_{\theta_i}^T F\kl y_j^t; \theta(0)\kr \el \theta_i(t) - \theta_i(0)\er.
		\end{align*}
		Using that, we obtain
		\begin{align*}
			&\quad\, \frac{2\lambda}{b_{\partial \Omega}} \sum_{j=1}^{b_{\partial \Omega}} F(y_j^t;\theta(t)) \sum_{i=1}^m \nabla_{\theta_i}^T F(y_j^t;\theta(t))\el \theta_i(t) - \overline{\theta_i}\er\\
			&\overset{\mathclap{\text{\cref{lemma:linearization error,lemma:approximation error,lemma:technical inequalities}}}}{\geq}\qquad\  \frac{2\lambda }{b_{\partial \Omega}} \sum_{j=1}^{b_{\partial \Omega}} F\kl y_j^t; \theta(t)\kr^2 - \frac{c_6\sqrt{\log \kl \frac{b_{\partial \Omega}}{\delta}\kr }}{\sqrt{m}} - c_7\norm{\eps_{p,\xi}}_{\infty(\partial \Omega)}
		\end{align*}
		with probability at least $1-\delta$. Here we applied \cref{lemma:approximation error} with $\tilde{\delta} = \frac{\delta}{b_{\partial \Omega}}$. Similarly, for (II), there holds
		\begin{align*}
			&\quad\, \sum_{i=1}^m \nabla_{\theta_i}^T \Delta F(y_j^t;\theta(t))\el \theta_i(t) - \overline{\theta_i}\er\notag \\
			&= \sum_{i=1}^m \nabla_{\theta_i}^T \Delta F\kl y_j^t;\theta(t)\kr \el \theta_i(t) - \theta_i(0)\er - \sum_{i=1}^m \nabla_{\theta_i}^T \Delta F\kl y_j^t; \theta(t)\kr \el \overline{\theta_i} - \theta_i(0)\er \notag \\
			&= \sum_{i=1}^m \nabla_{\theta_i}^T\Delta F\kl y_j^t; \theta(0)\kr \el \theta_i(t) - \theta_i(0)\er - \sum_{i=1}^m \nabla_{\theta_i}^T \Delta F\kl y_j^t; \theta(0)\kr \el \overline{\theta_i} - \theta_i(0)\er \notag\\
			&\quad\, + \sum_{i=1}^m \nabla_{\theta_i}^T \Delta F\kl y_j^t; \theta(0)\kr \el \overline{\theta_i} - \theta_i(0)\er - \sum_{i=1}^m \nabla_{\theta_i}^T\Delta F\kl y_j^t; \theta(t)\kr \el \overline{\theta_i} - \theta_i(0)\er \notag\\
			&\quad\, + \sum_{i=1}^m \nabla_{\theta_i}^T \Delta F\kl y_j^t;\theta(t)\kr \el \theta_i(t) - \theta_i(0)\er - \sum_{i=1}^m \nabla_{\theta_i}^T \Delta F\kl y_j^t; \theta(0)\kr \el \theta_i(t) - \theta_i(0)\er.
		\end{align*}
		Here we applied \cref{lemma:linearization error,lemma:approximation error,lemma:technical inequalities} with $\tilde{\delta} = \frac{\delta}{b_{\Omega}}$. Using that, we obtain
		\begin{align*}
			&\quad\, \frac{2}{b_\Omega}\sum_{j=1}^{b_\Omega} \kl \Delta F(x_j^t;\theta(t)) + f(x_j^t)\kr \sum_{i=1}^m \kl \Delta \nabla_{\theta_i} F(x_j^t;\theta(t))\kr^T \el \theta_i(t) - \overline{\theta_i}\er \\
			&\overset{\mathclap{\text{\cref{lemma:linearization error,lemma:approximation error,lemma:technical inequalities}}}}{\geq}\qquad \frac{2}{b_\Omega} \sum_{j=1}^{b_\Omega} \kl \Delta F\kl x_j^t;\theta(t)\kr + f(x_j^t)\kr^2 - \frac{c_8\sqrt{\log \kl \frac{b_\Omega}{\delta}\kr }}{\sqrt{m}} - c_9\norm{\Delta \eps_{p,\xi}}_{\infty(\Omega)}
		\end{align*}
		with probability at least $1-\delta$.
        
        \noindent\textbf{Combining upper and lower bound:} In total, this leads to
		\begin{align*}
			&\quad\, \sum_{i=1}^m \nabla_{\theta_i} E_{S_t}(\theta(t))^T \el \theta(t)-\overline{\theta}\er \\
            &\geq 2E_{S_t}\kl \theta(t)\kr -\frac{(c_6 + c_8)\sqrt{\log \kl \frac{\max \{b_\Omega, b_{\partial \Omega}\}}{\delta}\kr }}{\sqrt{m}} - c_{10} \norm{\eps_{p,\xi}}_{W^{0,2;\infty}}
		\end{align*}
		with probability at least $1-2\delta$. From \eqref{eq:lyapunov inequality}, we obtain
		\begin{align*}
			\mathcal{L}(W(t+1)) &\leq \underbrace{\sum_{i=1}^m \norm{\theta_i(t) - \overline{\theta_i}}_2^2}_{=\mathcal{L}(\theta(t))} - 4\eta E_{S_t}(\theta(t)) + \frac{2\eta(c_6+c_8)\sqrt{\log \kl \frac{\max \{b_\Omega, b_{\partial \Omega}\}}{\delta}\kr }}{\sqrt{m}} \\
            &\quad + \eta^2 (c_4 + c_5)^2 -4\eta c_{10}\norm{\eps_{p,\xi}}_{W^{0,2;\infty}}
		\end{align*}
		with probability $1-2\delta$. Now summing this inequality from $0$ to $T-1$ and dividing by $4\eta T$ implies
		\begin{align*}
            &\quad\, \frac{1}{T}\sum_{t=0}^{T-1} E_{S_t}(\theta(t)) \\
            &\leq \frac{\mathcal{L}(\theta(0))}{4\eta T} + \frac{(c_6+c_8)\sqrt{\log \kl \frac{\max \{b_\Omega, b_{\partial \Omega}\}}{\delta}\kr }}{2\sqrt{m}} + \frac{\eta\kl c_4 + c_5\kr^2}{4} + \norm{\eps_{p,\xi}}_{W^{0,2;\infty}}\\
			&\leq \frac{\alpha^2}{4\sqrt{T}} + \frac{(c_6 + c_8)\sqrt{\log \kl \frac{\max \{b_\Omega, b_{\partial \Omega}\}}{\delta}\kr }}{2\sqrt{m}} + \frac{(c_4 + c_5)^2}{4\sqrt{T}} +
   \eps_p + \xi
   .
		\end{align*}
 with probability 
  $1-2\delta$. Since $\xi$ was arbitrary, the claim follows by sending $\xi \to 0$.
\end{proof}

\subsection{A Computational Example}
In the last section we have formulated and proven \cref{thm:optimization bound}, which is a high-probability guarantee (with respect to the random initialization) showing an $\mathcal O(\frac1{\sqrt{T}} +\frac1{\sqrt{m}} + \epsilon_p)$ bound on the empirical risk. 
Here $\epsilon_p$ resembles the approximation error of the true solution of the PDE. 
To compare this to practical applications, we conduct an experiment on a toy problem.
To this end, we choose $p$ large such that we can expect $\epsilon_p$ to be small and set $T=m$ and train shallow networks of different widths $m$ for $T=m$ iterations. 
We expect an (almost) $\mathcal O(\frac{1}{\sqrt{T}})$ decay of the empirical loss. 
To this end, we consider the two-dimensional Poisson equation 
\[
\begin{aligned}
            -\Delta u&=1 && \text{in }\Omega \subset \R^2 \\
            u&=0 && \text{on }\partial\Omega
\end{aligned},
\]
on the disc 
\[
     \Omega = B_2\kl \begin{pmatrix}
    0 \\ 2
    \end{pmatrix}, 
    1\kr.
\]
\begin{figure}[h]
    \centering
    \hspace{.3cm}
    \begin{tikzpicture}
    \begin{loglogaxis}[
        width=10cm,
        height=7cm,
        xlabel={$m = T$},
        ylabel={$\displaystyle \text{90\% percentile of } \frac{1}{T} \sum_{t=0}^{T-1} E_{S_t}(\theta(t))$},
        xlabel style={font=\small},
        ylabel style={font=\small},
        legend style={font=\small, at={(0.99,0.99)}, anchor=north east},
        log ticks with fixed point, 
        mark size=2.5pt,
        grid=both,
        xtick={5,10,20,40},
        ytick={0.5,1,2,4},
        xmin = 3,
        xmax = 35,
        ymin = 0.9,
        ymax = 6,
    ]
    
    \addplot+[only marks, mark=*, color=UltraMarine] 
        coordinates {(4,4.575) (6,3.250) (10,2.136) (14,1.716) (20,1.424) (24,1.293)  (30,1.159)};
    
    \addlegendentry{90\% percentile bound}
    
    \addplot+[domain=3:40, samples=200, mark=none, very thick, color=black]{6.33 * x^(-0.5)};
    
    \addlegendentry{Slope $-0.5$}
    
    \end{loglogaxis}
    \end{tikzpicture}
    \caption{Shown are the $90\%$ percentile of the empirical loss compared to different network widths $m$ trained with projected stochastic gradient descent for $T=m$ steps as well as the $\mathcal O\bigl(\frac1{\sqrt{T}}\bigr) = \mathcal O\bigl(\frac1{\sqrt{m}}\bigr)$ rate, which is guaranteed by \Cref{thm:optimization bound} up to an approximation error.}
    \label{fig:experiment uniform data points}
\end{figure}
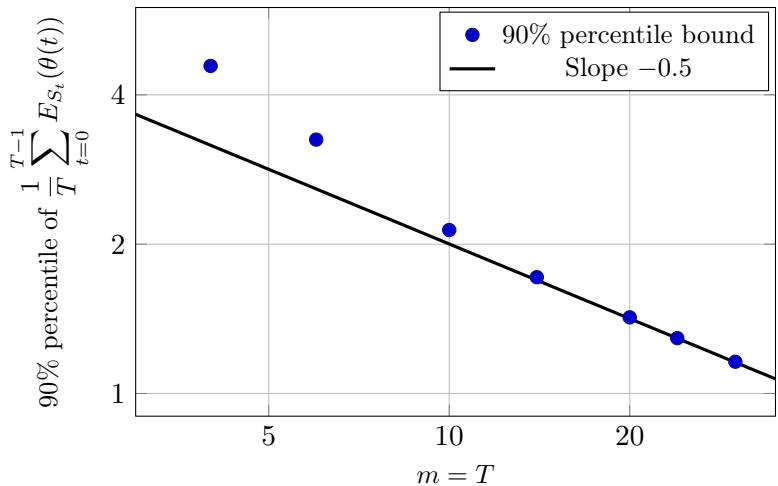
We choose $\sigma(x) = \tanh(x)$ as an activation function, which satisfies our regularity \Cref{ass:smoothness}. 
In the training processes, we choose SGD with step size $\eta = \frac{1}{\sqrt{T}}$, corresponding to batch sizes $b_\Omega=1, b_{\partial \Omega} = 1$ and choose a projection radius of $p = 40$. 
For the sampling of the integration points, we take independent samples from the uniform distributions on $\Omega$ and $\partial\Omega$, respectively. 
As network sizes, we consider $m=4,6,10,14,20,24,30$. 
As \Cref{thm:optimization bound} gives a high-probability guarantee, we report the $90\%$ percentile of time average $\frac1T\sum_{t=0}^{T-1} E_{S_t}(\theta(t))$ of the empirical loss in \Cref{fig:experiment uniform data points} in a log-log plot, where we also indicate our idealized theoretical $\mathcal O\bigl(\frac1{\sqrt{T}}\bigr)$ guarantee. 
To approximate the percentile, we use $10^6$ runs for $m=4,6$ and $10^5$ for the other network sizes. 
We observe that the percentiles approximately nestle around the $\mathcal O\bigl( \frac1{\sqrt{T}}\bigr) = \mathcal O\bigl( \frac1{\sqrt{m}}\bigr)$ rate that we obtain from \Cref{thm:optimization bound} up to an approximation error.

\section{Overall Error Analysis} 
In this section, we combine \cref{thm:optimization bound} together with a generalization estimate into a total error bound.
From now on, we assume that the sampling of data points follows a uniform distribution on $\Omega$ as well as on $\partial \Omega$. We make this particular choice of the sampling distribution because it is widely used in practice. Nonetheless, we emphasize that one can analyze the generalization properties of other distributions as well, since our optimization result \cref{thm:optimization bound} holds for any distribution. In the end, we will get the following result on the overall error.

\begin{thm}[Complete error analysis]\label{thm:overall-error}
    For $T\in\N$ we denote the optimization step with the best empirical loss $\hat{t} \coloneqq \argmin_{t=0, \dots, T} E_{S_t}(\theta(t))$ and set $\hat{\theta}_T\coloneqq \theta(\hat{t})$.
    Given the presumption that we choose iid uniformly distributed data points $S_t$ and that we consider the setting of \cref{thm:optimization bound}, there holds
\begin{equation}
	\norm{
 F(\cdot,{\hat\theta_T}) 
 - u^*}_{H^{\frac{1}{2}}(\Omega)}^2 \precsim \kl \frac{1}{\sqrt{T}} + \frac{\sqrt{\log \kl \frac{\max \{b_\Omega, b_{\partial \Omega}\}}{\delta}\kr }}{\sqrt{m}}+ \eps_{p} + \sqrt{\frac{\log \frac{1}{\delta}}{\min \{b_\Omega, b_{\partial \Omega}\}}}\kr^{\frac{1}{2}}
\end{equation}
with probability at least $1-3\delta$.
\end{thm}
To prove \cref{thm:overall-error}, we are particularly interested in the best-iterate generalization error. To that end, we will prove a bound using Rademacher complexity estimates and a version of MCDiarmid's inequality. New in that regard is the necessity to estimate the Rademacher complexity of the Laplacian of the function class. Hence, we now derive a bound for it before stating the generalization theorem and giving the proof of \cref{thm:overall-error} in the end of this section. For a fixed $m\in \N_e$, we set
\begin{align}
    \mathcal{F}_m^p \coloneqq \left\{
        x\mapsto F(x;\theta) \mid \theta\in B_{2,\infty}\kl \theta(0), \frac{p}{\sqrt{m}}\kr
    \right\}
\end{align}
and $\Delta (\mathcal{F}_m^p) = \{\Delta u \mid u\in \mathcal{F}_m^p\}$. 
\begin{prop}[Estimate on Rademacher complexity]
    There holds
    \begin{align}
        \rademp_{\min\{b_\Omega, b_{\partial \Omega}\}} (\Delta(\mathcal{F}_m^p)) \precsim \frac{1}{\sqrt{m}} + \frac{1}{\sqrt{\min\{b_\Omega, b_{\partial \Omega}\}}}. 
    \end{align}
\end{prop}
\begin{proof}
Our strategy is the same as in \cref{lemma:rademacher complexity nn}: We decompose the function class $\mathcal{F}_m^p$ into its linearization and a remainder and bound the Rademacher complexity of both parts individually. In the proof of \cref{thm:optimization bound} we have shown that $\lVert{\nabla_{\theta_i}\Delta F(x;\theta(t))}\rVert_2 \precsim \frac{1}{\sqrt{m}}$. 
We define $B\coloneqq B_{2,\infty} ( \theta(0), \frac{p}{\sqrt{m}})$, $n \coloneqq \min\{b_\Omega, b_{\partial \Omega}\}$ and compute
\begin{align*}
    &\quad\, n \rademp_n(\Delta(\mathcal{F}_m^p)) \\
    &= \E_\varepsilon \el \sup_{\theta\in B} \sum_{j=1}^n \varepsilon_j \Delta F(x_j; \theta, c)\er \\
    &= \E_\varepsilon \el \sup_{\theta\in B} \sum_{j=1}^n \varepsilon_j \el \Delta F(x_j; \theta, c) - \Delta F(x_j; \theta(0), c) \er\er \\
    &\precsim \frac{n}{\sqrt{m}} + \E_\varepsilon \el \sup_{\theta\in B} \sum_{j=1}^n \varepsilon_j \nabla^T \Delta F(x_j; \theta(0), c)\el \theta- \theta(0)\er\er \\
    &\precsim \frac{n}{\sqrt{m}} + \E_\varepsilon \el \sup_{w\in B} \underbrace{\norm{\theta-\theta(0)}_2}_{\leq p} \norm{\sum_{j=1}^n \varepsilon_j \nabla^T \Delta F(x_j; \theta(0), c)}_2 \er\\
    &\precsim \frac{n}{\sqrt{m}} + \E_\varepsilon \el \sup_{\theta\in B} \underbrace{\norm{\theta-\theta(0)}_2}_{\leq p} \norm{\sum_{j=1}^n \varepsilon_j \sum_{i=1}^m \nabla_{\theta_i} \Delta F(x_j; \theta(0), c)}_2 \er\\
    &\precsim \frac{n}{\sqrt{m}} + p \E_\varepsilon \el \norm{\sum_{j=1}^n \varepsilon_j \sum_{i=1}^m \nabla_{\theta_i} \Delta F(x_j; \theta(0), c)}_2 \er\\
    &\precsim \frac{n}{\sqrt{m}} + p \sqrt{\sum_{j=1}^n \norm{\sum_{i=1}^m \nabla_{\theta_i} \Delta F(x_j; \theta(0), c)}_2^2} \\
    &\precsim \frac{n}{\sqrt{m}} +  p\sqrt{n}.
\end{align*}
\end{proof}
Now we are ready to state and prove our result on a particular generalization error. In our analysis, this error measures how far away the best empirical loss function measurement is from the exact loss function at the best empirical loss weights. 
\begin{prop}[Generalization Bound under Uniform Sampling]\label{thm:generalization bounds}Let $\hat{t} \coloneqq \argmin_{t} E_{S_t}(\theta(t))$ denote the optimization step with the best empirical loss. Given the presumption that we choose iid uniformly distributed data points, there exist constants $C_4, C_5>0$ such that for every $\delta > 0$ there holds
    \begin{align}
        \betrag{E_{S_{\hat{t}}}(\theta(\hat{t})) - E(\theta(\hat{t}))} &\leq \frac{C_4}{\sqrt{m}} + C_5\sqrt{\frac{\log \frac{1}{\delta}}{\min\{b_\Omega, b_{\partial \Omega}\}}}
    \end{align}
    with probability of at least $1-\delta$.
\end{prop}
In contrast to the classical supervised learning setting, the Laplacian is part of the loss function. Using an estimation on the Rademacher complexity of the Laplacian of the neural network, we can conclude a generalization bound by combining these results with classical Rademacher complexity bounds \cite{bartlett02a}.
\begin{proof}
    We start by stating a version of McDiarmid's inequality \cite{McDiarmid1989} in terms of the empirical Rademacher complexity. This is, there exists a constant $c$ such that
    \[
        \betrag{E_{S_{\hat{t}}}(\theta(\hat{t})) - E(\theta(\hat{t}))} \leq 2 \rademp_{\min\{b_\Omega, b_{\partial \Omega}\}} (E(\mathcal{F}_m^p)) + c \sqrt{\frac{\log \frac{1}{\delta}}{\min\{b_\Omega, b_{\partial \Omega}\}}}.
    \]
    It now remains to bound the empirical Rademacher complexity. Since both $\rademp_{\min\{b_\Omega, b_{\partial \Omega}\}} (\mathcal{F}_m^p)$ and $\rademp_{\min\{b_\Omega, b_{\partial \Omega}\}} (\Delta(\mathcal{F}_m^p))$ are bounded by a constant multiple of $\frac{1}{\sqrt{m}} + \frac{1}{\sqrt{\min\{b_\Omega, b_{\partial \Omega}\}}}$ (the proof for the first part is similar to the second and is done explicitly in \cref{lemma:rademacher complexity nn}), we can bound the empirical Rademacher complexity of $(\Delta F + f)^2$ and $F^2$ in the loss function by using Talagrand's contraction principle \cite{Talagrand91}. To this end, we observe that the Lipschitz constant of $\R\supsetneq \Omega \to \R, x\mapsto x^2$ is given by $2\sup_{x\in \Omega} \betrag{x}$. In total, there holds
    \[
        \rademp_{\min\{b_\Omega, b_{\partial \Omega}\}} (E(\mathcal{F}_m^p)) \precsim \frac{1}{\sqrt{m}} + \frac{1}{\sqrt{\min\{b_\Omega, b_{\partial \Omega}\}}}
    \]
    and thus the result follows.
\end{proof}
We now combine \cref{thm:optimization bound,thm:generalization bounds} to prove \cref{thm:overall-error}.
\begin{proof}[Proof of \Cref{thm:overall-error}]
Throughout this work, we have proven bounds for the optimization error (\cref{thm:optimization bound}) and the generalization error (\cref{thm:generalization bounds}), which enables us to connect the empirical loss function to the exact loss function. It remains now to understand the relationship between the value of the exact loss function at a given $\theta$ and the distance between $u_\theta$ and $u^*$ with respect to a suitable norm. To this end, we make use of \cite[Theorem 8]{Mueller2022} and decompose the error into
\[
    \norm{F(\cdot; \theta(\hat{t})) - u^*}_{H^s(\Omega)} \leq c_{\text{reg}}\sqrt{\eps_{\text{generalization}} + \eps_{\text{optimization}}}
\]
with
\[
    \eps_{\text{generalization}} = \abs{E(\theta(\hat{t})) - E_{S_{\hat{t}}}(\theta(\hat{t})) }, \quad
    \eps_{\text{optimization}} = E_{S_{\hat{t}}}(\theta(\hat{t})),
\]
since $E(\theta(\hat{t})) \leq \eps_{\text{generalization}} + \eps_{\text{optimization}}$ and thus we obtain a complete high-probability error bound.     
\end{proof}

\section*{Acknowledgements}
The authors want to thank Semih \c{C}ayc{\i} for helpful discussions and advice regarding this work. 
JN acknowledges funding by the Deutsche Forschungsgemeinschaft (DFG, German Research Foundation) – 320021702/GRK2326 – Energy, Entropy, and Dissipative Dynamics (EDDy).
JM acknowledges funding by the Deutsche Forschungsgemeinschaft (DFG, German Research Foundation) under the project number 442047500 through the Collaborative Research Center Sparsity and Singular Structures (SFB 1481).

\appendix
\section*{Appendix}

We initialize the weights according to the following distribution. 

\begin{definition}[Truncated Normal Distribution]\label{definition: truncated normal distribution}
			Let $\mathcal{N}(\mu, \sigma^2)$ be a normal distributed random variable with mean $\mu$ and variance $\sigma^2$, density $\varphi$ and cumulative distribution function $\Phi$. Then the corresponding truncated normal distribution $\mathcal{N}_{(a,b)}(\mu, \sigma^2)$ on the interval $(a,b)$ is defined via the probability density function
			\[
			f(x; \mu, \sigma, a, b) \coloneqq \begin{cases}
				\frac{1}{\sigma} \frac{\varphi\kl \frac{x-\mu}{\sigma}\kr }{\Phi\kl \frac{b-\mu}{\sigma}\kr - \Phi\kl \frac{a-\mu}{\sigma}\kr} & \text{for } x\in (a,b) \\
				0 & \text{otherwise}
			\end{cases}.
			\]
			For $a>0$ we set $\mathcal{N}_a(\mu, \sigma^2) \coloneqq \mathcal{N}_{(-a,a)}(\mu, \sigma^2)$. The multivariate truncated normal distribution is obtained by taking the product measure of one-dimensional truncated normal distributions.
\end{definition}

Now, we provide the proofs of the linearization and approximation results. 

\begin{proof}[Proof of \cref{lemma:linearization error}]~\\
    \begin{enumerate}
            \item Using the triangle inequality and the fact that $\Delta F(x;\theta(0)) = 0$ for all $x\in \Omega$, we compute
			\begin{align*}
				&\quad\, \betrag{\sum_{i=1}^m \nabla_{\theta_i}^T \Delta F\kl x; \theta(0) \kr \el \theta_i - \theta_i(0)\er - \Delta F\kl x; \theta(t)\kr} \\
				&= \betrag{\sum_{i=1}^m \nabla_{\theta_i}^T \Delta F\kl x; \theta(0) \kr \el \theta_i - \theta_i(0)\er - \frac{1}{\sqrt{m}} \sum_{i=1}^m \sum_{j=1}^d c_i \theta_{ij}^2 \sigma''\kl \theta_i^T x\kr }\\
				&= \left| \frac{1}{\sqrt{m}} \sum_{i=1}^m c_i \begin{bmatrix}
					\el \sum_{l=1}^d x_1 \theta_{il}^2(0) \sigma'''\kl \theta_i(0)^T x\kr \er + 2\theta_{i1}(0) \sigma''\kl \theta_i(0)^T x\kr \\
					\vdots \\
					\el \sum_{l=1}^d x_d \theta_{il}^2(0) \sigma'''\kl \theta_i(0)^T x\kr \er + 2\theta_{id}(0) \sigma''\kl \theta_i(0)^T x\kr
				\end{bmatrix}^T \el \theta_i - \theta_i(0) \er \right.\\
				&\quad\, \left. - \frac{1}{\sqrt{m}} \sum_{i=1}^m \sum_{j=1}^d c_i \theta_{ij}^2 \sigma''\kl \theta_i^T x\kr \right| \\
				&=\left|\frac{1}{\sqrt{m}} \sum_{i=1}^m c_i \el \sum_{l=1}^d \sum_{j=1}^d x_l \theta_{ij}^2(0) \sigma'''\kl \theta_i(0)^T x\kr \el \theta_{il}(t) - \theta_{il}(0)\er\right. \right. \\
				&\quad\,\left.\left. + \sum_{l=1}^d 2\theta_{il}(0) \sigma''\kl \theta_i(0)^T x\kr \el \theta_{il}(t) - \theta_{il}(0)\er - \sum_{j=1}^d \theta_{il}^2(t) \sigma''\kl \theta_i^T x\kr \er\right| \\
				&=\left|\frac{1}{\sqrt{m}} \sum_{i=1}^m c_i \el \sum_{j=1}^d \bigg[ \theta_{ij}^2 \sigma''\kl \theta_i^T x\kr - 2\theta_{ij} \theta_{ij}(0) \sigma''\kl \theta_i(0)^T x\kr + 2\theta_{ij}^2(0) \sigma''\kl \theta_i(0)^T x\kr\right.\right. \\
				&\quad\,  - \theta_{ij}^2(0) \sigma'''\kl \theta_i(0)^T x\kr \el x^T \kl \theta_i - \theta_i(0)\kr \er \bigg] \Biggr] \Bigg| \\
				&\leq \left|\frac{1}{\sqrt{m}} \sum_{i=1}^m c_i \el \sum_{j=1}^d \bigg[ \kl \theta_{ij} - \theta_{ij}(0)\kr^2 \sigma''\kl \theta_i(0)^T x\kr + \theta_{ij}^2(0) \sigma''\kl \theta_i(0)^T x\kr \right.\right. \\
				&\quad\, +  \theta_{ij}^2 \sigma''\kl \theta_i^T x\kr - \theta_{ij}^2 \sigma''\kl \theta_i(0)^T x\kr - \theta_{ij}^2(0) \sigma'''\kl \theta_i(0)^T x\kr \el x^T \kl \theta_i - \theta_i(0)\kr \er \bigg] \Biggr] \Bigg| \\
				&= \left|\frac{1}{\sqrt{m}} \sum_{i=1}^m c_i \el \sum_{j=1}^d \bigg[ \underbrace{\kl \theta_{ij} - \theta_{ij}(0)\kr^2}_{\precsim \frac{1}{m}} \sigma''\kl \theta_i(0)^T x\kr \right.\right. \\
				&\quad\,  \underbrace{+\theta_{ij}^2 \sigma''\kl \theta_i^T x\kr - \theta_{ij}^2 \sigma''\kl \theta_i(0)^T x\kr - \theta_{ij}^2(0) \sigma'''\kl \theta_i(0)^T x\kr \el x^T \kl \theta_i - \theta_i(0)\kr \er}_{\overset{\mathclap{\text{\cref{lemma:application of triangle inequality}}}}{\precsim}\qquad\  \frac{1}{m}} \bigg] \Biggr] \Bigg| \\
				&\precsim \frac{1}{\sqrt{m}} .
			\end{align*}
            Here we used that for any $\beta$-smooth function $\sigma$ and any $r,s$, there holds
			\[
				\betrag{\sigma(r) - \sigma(s) - \sigma'(s)(r-s)} = \betrag{\int_r^s \kl \sigma'(t) - \sigma'(s)\kr \intd{t}} \leq \frac{\beta (r-s)^2}{2}.
			\]
            We will use this inequality from now on without explicitly repeating the argument.
			\item The proof, as well as the statement, can be found in \cite[Proposition 3.4]{Telgarsky2023}.
   \end{enumerate}
   \end{proof}
   \begin{proof}[Proof of \cref{lemma:approximation error}]
    \begin{enumerate}
			\item First, we observe that
			\[
			\Delta u(x) = \E_{\theta_0\sim \mathcal{N}_a(0,I_d)} \el \sum_{k=1}^d 2v_k(\theta_0) \theta_{0_k} \sigma''(\theta_0^T x) + \sum_{k=1}^d \sum_{j=1}^d v_j (\theta_0) x_j \theta_{0_k}^2 \sigma'''(\theta_0^T x)\er . 
			\]
			This implies
			\begin{align*}
				&\quad\, \sum_{i=1}^m \nabla_{\theta_i}^T \Delta F\kl x; \theta(0) \kr \el \theta_i - \theta_i(0)\er \\
				&=  \frac{1}{\sqrt{m}} \sum_{i=1}^m c_i \el \sum_{k=1}^d \sum_{j=1}^d x_j \theta_{ik}^2(0) \sigma'''\kl \theta_i(0)^T x\kr U_{ij} + \sum_{k=1}^d 2 \theta_{ik}(0) \sigma''\kl \theta_i(0)^T x\kr U_{ik} \er \\
				&= \frac{1}{m} \sum_{i=1}^m \el \sum_{k=1}^d \sum_{j=1}^d v_j\kl \theta_i(0)\kr x_j \theta_{ik}^2(0) \sigma'''\kl \theta_i(0)^T x\kr + \sum_{k=1}^d 2v_k\kl \theta_i(0)\kr \theta_{ik}(0) \sigma''\kl \theta_i(0)^T x\kr \er \\
				&\xrightarrow{m\to\infty} \Delta u(x)\text{ almost surely}.
			\end{align*}
			Thus from Hoeffding's inequality \cite{boucheron2003concentration} and the triangle inequality the claim follows.
            \item We compute
			\begin{align*}
				&\quad\, \sum_{i=1}^m \nabla_{\theta_i}^T F\kl x; \theta(0) \kr \el \theta_i - \theta_i(0)\er \\
				&= \frac{1}{\sqrt{m}}\sum_{i=1}^m c_i x^T \sigma'\kl \theta_i(0)^T x\kr \frac{1}{\sqrt{m}} c_i v\kl \theta_i(0)\kr \\
				&= \frac{1}{m} \sum_{i=1}^m v^T(\theta_i(0))x \sigma'\kl \theta_i(0)^T x\kr
			\end{align*}
			and observe that $v^T(\theta_i(0))x\sigma'(\theta_i(0)^T x)$ is an unbiased Monte Carlo estimator of $u(x)$ (for each $i\in \{1,\ldots, m\}$) and thus the expression in the last line converges to $u(x)$ almost surely for $m\to \infty$. Thus from Hoeffding's inequality and the triangle inequality the claim follows.
		\end{enumerate}
\end{proof}

\begin{restatable}{lem}{technicalinequalities}\label{lemma:technical inequalities}
Let $m\in \N_e$, fix $\theta(0)$ sampled from \cref{eq: ntk initialization} and let $\theta_i \in B_{2,\infty}\kl \theta(0), \frac{p}{\sqrt{m}}\kr$. Then the following statements hold: For all $x\in \partial \Omega$ we have
\begin{align}
        \betrag{\sum_{i=1}^m \nabla_{\theta_i}^T F\kl x; \theta\kr \el \theta_i - \theta_i(0)\er - \sum_{i=1}^m \nabla_{\theta_i}^T F\kl x; \theta(0)\kr \el \theta_i - \theta_i(0)\er} & \precsim \frac{1}{\sqrt{m}}. 
\end{align}
For all $x\in \Omega$ we have
\begin{align}
    \betrag{\sum_{i=1}^m \nabla_{\theta_i}^T \Delta F\kl x; \theta\kr \el \theta_i - \theta_i(0)\er - \sum_{i=1}^m \nabla_{\theta_i}^T \Delta F\kl x; \theta(0)\kr \el \theta_i - \theta_i(0)\er} & \precsim \frac{1}{\sqrt{m}}.
\end{align}
\end{restatable}
\begin{proof}
\textbf{Estimates for $\partial \Omega$:}
    We calculate
			\begin{align*}
				&\quad\, \betrag{\sum_{i=1}^m \nabla_{\theta_i}^T F\kl x; \theta(t)\kr \el \theta_i - \theta_i(0)\er - \sum_{i=1}^m \nabla_{\theta_i}^T F\kl x; \theta(0)\kr \el \theta_i - \theta_i(0)\er} \\
				&\leq \frac{1}{m}\sum_{i=1}^m p C_{\Omega} \underbrace{\betrag{\sigma'\kl \theta_i(t)^T x\kr - \sigma'\kl \theta_i(0)^T x\kr}}_{\begin{aligned}
						&\leq \sigma_2\betrag{\el \theta_i(t) - \theta_i(0)\er^T x} \\
						&\leq \sigma_2 C_{\Omega} \norm{\theta_i(t)-\theta_i(0)}_2 \\
						&\leq \frac{\sigma_2 C_{\Omega} p}{\sqrt{m}}
				\end{aligned}} \\
				&\leq \frac{p^2 C_{\Omega}^2 \sigma_2}{\sqrt{m}}.
			\end{align*}
\textbf{Estimates for $\Omega$:}
There holds
			\begin{align*}
			&\quad\, \betrag{\sum_{i=1}^m \nabla_{\theta_i}^T \Delta F\kl x; \theta(t)\kr \el \theta_i - \theta_i(0)\er - \sum_{i=1}^m \nabla_{\theta_i}^T \Delta F\kl x; \theta(0)\kr \el \theta_i - \theta_i(0)\er} \\
			&= \betrag{\sum_{i=1}^m \el \nabla_{\theta_i}^T \Delta F\kl x; \theta(t)\kr - \nabla_{\theta_i}^T \Delta F\kl x;\theta(0)\kr \er \el \theta_i - \theta_i(0)\er} \\
			&= \left|\frac{1}{\sqrt{m}} \sum_{i=1}^m c_i\right. \\
			&\quad\, \times \left.\begin{bmatrix}
			\el \sum_{l=1}^d x_1 \el \theta_{il}^2(t) \sigma'''\kl \theta_i(t)^T x\kr - \theta_{il}^2(0) \sigma'''\kl \theta_i(0)^T x\kr\er \er \\ + \el 2\theta_{i1}(t) \sigma''\kl \theta_i(t)^T x\kr - 2\theta_{i1}(0) \sigma''\kl \theta_i(0)^T x\kr\er \\
			\vdots \\
			\el \sum_{l=1}^d x_d \el \theta_{il}^2(t) \sigma'''\kl \theta_i(t)^T x\kr - \theta_{il}^2(0) \sigma'''\kl \theta_i(0)^T x\kr\er \er \\ + \el 2\theta_{id}(t) \sigma''\kl \theta_i(t)^T x\kr - 2\theta_{id}(0) \sigma''\kl \theta_i(0)^T x\kr\er
			\end{bmatrix}^T \el \theta_i - \theta_i(0)\er\right| \\
			&\overset{\mathclap{\text{\cref{lemma:application of triangle inequality}}}}{\underset{\mathclap{\text{Cauchy-Schwarz ineq.}}}{\precsim}} \qquad \frac{1}{\sqrt{m}}.
			\end{align*}
\end{proof}
\begin{lem}[Rademacher Complexity of Shallow Networks]\label{lemma:rademacher complexity nn}
    We consider a neural network $F(x;\theta,c) = \frac{1}{\sqrt{m}}\sum_{i=1}^m c_i \sigma(\theta_i^T x)$ and let the activation function satisfy  
    \[
        \begin{aligned}
            \abs{\sigma'(z)-\sigma'(\tilde{z})} &\leq \beta \abs{z-\tilde{z}} \quad \text{for all } z, \tilde{z} \in \R \quad\text{and}\\
            \sup_z \abs{\sigma'(z)} &\leq \kappa. 
        \end{aligned}
    \]
    Further let
    \[
        \mathcal{F}_m^p \coloneqq \left\{
            x\mapsto F(x;\theta,c) \mid \theta\in
            B_{2,\infty}\kl \theta(0), \frac{p}{\sqrt{m}}\kr
        \right\}.
    \]
    We assume that $\norm{x_j}_2 \leq 1$ for $j=1,2, \ldots, n$. Then there holds
    \[
        \rademp_n(\mathcal{F}_m^p) \leq \frac{p^2 \beta}{2\sqrt{m}} + \frac{\kappa p}{\sqrt{n}}
    \]
    for any $p>0$ and $m,n\in \N$.
\end{lem}
\begin{proof}
    We will only state the proof strategy: 
    We decompose the function class $\mathcal F^p_m$ into its linearization and a remainder and bound the Rademacher complexity of both parts individually. 
    Since we know from \cite{bach2024learning} that $\rademp_n(\mathcal{F}_m^p) \le \rademp_n({\mathcal{F}_m^{p,\textup{lin}}}) + \eps_{\textup{lin}}$, where $\eps_{\textup{lin}}$ denotes the linearization error as in \cref{lemma:linearization error}, we can use \cref{lemma:linearization error} to get the result together with the result $\rademp_n({\mathcal{F}_m^{p,\textup{lin}}}) \le \frac{\kappa p}{\sqrt{n}}$ from \cite{bach2024learning} to get the result.
\end{proof}

\begin{lem}\label{lemma:application of triangle inequality}
Let $a,b,c,d\in \R$.
There holds
\[
\betrag{ab-cd} \leq \betrag{a}\betrag{b-d} + \betrag{d}\betrag{a-c}.
\]
Assume further that $(a-b)^2\leq \eps$ for some $\eps > 0$. Then
\[
    \betrag{a^2-b^2} \leq \betrag{a+b}\sqrt{\eps}.
\]
\end{lem}
\begin{proof}
We use the triangle inequality to obtain
    \[
    \betrag{ab-cd} = \betrag{ab - ad + ad - cd} \leq \betrag{ab-ad} + \betrag{ad-cd} = \betrag{a} \betrag{b-d} + \betrag{d} \betrag{a-c}.
    \]
Further, we compute
    \[
        \betrag{a^2-b^2} = \betrag{a-b}\betrag{a+b} \leq \sqrt{\eps} \betrag{a+b}.
    \]
\end{proof}

\printbibliography

\end{document}